\documentclass[]{amsart}
\usepackage[utf8]{inputenc}
\usepackage[T1]{fontenc}
\usepackage{amsmath,amssymb, amsthm, graphicx, tikz, mathtools, mathrsfs, comment, setspace, microtype,dsfont,float,hyperref,tikz-cd,enumitem,accents,xspace,bm,subcaption}
\usepackage[justification=centering]{caption}

\usepackage{biblatex}
\newtheorem{theorem}{Theorem}[section]
\newtheorem{corollary}[theorem]{Corollary}
\newtheorem{lemma}[theorem]{Lemma}
\newtheorem{proposition}[theorem]{Proposition}
\newtheorem{definition}[theorem]{Definition}

\renewcommand{\ker}{\text{\textnormal{ker}}}
\newcommand{\R}{\mathbb{R}}
\newcommand{\C}{\mathbb{C}}

\newcommand{\Abs}[1]{\left\lvert #1 \right\rvert}

\newcommand{\brackets}[1]{\langle #1 \rangle}

\newcommand{\dsum}{\oplus}

\newcommand{\interior}[1]{\accentset{\circ}{#1}}

\renewcommand{\Re}{\text{\textnormal{Re}}}
\renewcommand{\Im}{\text{\textnormal{Im}}}
\newcommand{\norm}[1]{\left\lVert#1\right\rVert}

\newcommand{\tr}{\text{\textnormal{tr}}}

\newcommand{\dist}{\text{\textnormal{dist}}}
\newcommand\restr[2]{{
		\left.\kern-\nulldelimiterspace 
		#1 
		\vphantom{\big|} 
		\right|_{#2} 
}}

\begin{document}
	\title{Convergence of the Laws of Non-Hermitian Sums of Projections}
	\author{Max Sun Zhou}	
	\date{}
	
	\begin{abstract}
		We consider the random matrix model $X_n = P_n + i Q_n$, where $P_n$ and $Q_n$ are independently Haar-unitary rotated Hermitian matrices with at most $2$ atoms in their spectra. Let $(M, \tau)$ be a tracial von Neumann algebra and let $p, q \in (M, \tau)$, where $p$ and $q$ are Hermitian and freely independent. Our main result is the following convergence result: if the law of $P_n$ converges to the law of $p$ and the law of $Q_n$ converges to the law of $q$, then the empirical spectral distributions of the $X_n$ converges to the Brown measure of $X = p + i q$. To prove this, we use the Hermitization technique introduced by Girko, along with the algebraic properties of projections to prove the key estimate. We also prove a converse statement by using the properties of the Brown measure of $X$.
	\end{abstract}	
	
	\maketitle

\section{Introduction}

Let $M$ be a von Neumann algebra with a faithful, normal tracial state $\tau$. The pioneering work of Voiculescu in \cite{Voiculescu1991} showed it is possible to understand the limits of empirical spectral distributions of independent Gaussian random matrices as the laws of freely independent random variables. While there are powerful and varied techniques in determining the limit laws of Hermitian random matrices, there have been few techniques and examples for determining the limit laws of non-Hermitian (i.e. non-normal) random matrices. Two notable examples have been the Circular Law (finally proven in complete generality in \cite{TaoPaper}) and the Single Ring Theorem in \cite{GuionnetPaper} (which has been recently generalized in \cite{ZhongHoPaper}).

Before proving the proving the convergence, we must first determine what the limit of the random matrix model should even be. Of the random matrix models that are considered, typically there is already a natural limit operator of the random matrix model that comes from free probability. To any operator in a tracial von Neumann algebra, there is an associated complex Borel probability measure, the \textit{Brown measure}. The Brown measure of the natural limit operator is the candidate limit law of empirical spectral distributions of the random matrices. There are counterexamples to this (see  \cite[][Chapter 11, Exercise 15]{SpeicherBook} for a simple counterexample), but typically this is true (see \cite{SniadyPaper} for a precise statement of this). In the case of the Circular Law and the Single Ring Theorem, the limit law is indeed the Brown measure of the limit operator. 

Having identified the tentative limit law, then we can ask what techniques are available in proving the convergence. The main method was first introduced in \cite{GirkoPaper} as a method of proving the Circular Law and involves a  ``Hermitization'' of the problem, i.e. translating the problem of the convergence of the non-Hermitian random matrices to the convergence of associated Hermitian random matrices. The convergence of these associated Hermitian random matrices is well-understood, but the trade-off is that one must bound the minimum singular values of the random matrices from below. The Hermitization method was used to great effect in the proofs of the Circular Law and the Single Ring Theorem, and obtaining the appropriate minimum singular value estimates was the critical step in both of the proofs. 

In this paper, we will use this method and prove that the limit of the empirical spectral distributions of a random matrix model $X_n$ is the Brown measure of the appropriate limit operator. 

Now, we define the matrix model $X_n$ that is considered in this paper:

\begin{definition}
	\label{def:X_n}
	Let $\mathcal{H}_n$ denote the Haar measure on the unitary group $U(n) \subset M_n(\C)$. Let $\{U_n, V_n\}_{n \geq 1}$ be a sequence of independent, $\mathcal{H}_n$-distributed matrices. Let $P_n', Q_n' \in M_n(\C)$ be deterministic, Hermitian, and 
	\begin{equation}
		\label{eqn:x_n_defn}
		\begin{aligned}
			\bm{\mu_{P_n'}} & = a_n \delta_{\alpha_n} + (1 - a_n) \delta_{\alpha_n'} \\
			\bm{\mu_{Q_n'}} &= b_n \delta_{\beta_n} + (1 - b_n) \delta_{\beta_n'} \, ,
		\end{aligned}
	\end{equation}
	for $a_n, b_n \in [0, 1]$, $\alpha_n, \alpha_n', \beta_n, \beta_n' \in \R$. Let
	\begin{equation}
		\begin{aligned}
			\bm{P_n} & = U_n P_n' U_n^* \\
			\bm{Q_n} &= V_n Q_n' V_n^* \,.
		\end{aligned}
	\end{equation}
	Then, define the matrix model $\bm{X_n}$ as: 
	\begin{equation}
		\bm{X_n} = P_n + i Q_n \,.
	\end{equation}
\end{definition}

We previously considered the natural limit operator of the $X_n$ in \cite{BrownMeasurePaper1}: $X = p + i q$, where $p, q \in (M, \tau)$ are Hermitian, freely independent, and have spectra consisting of 2 atoms. In that paper, we explicitly computed the Brown measure of $X$. In this paper, we will prove the convergence of the empirical spectral distributions of the $X_n$ to the Brown measure of $X$. 

Our main result is Theorem \ref{thm:convergence_brown_measure}: 

\newtheorem*{thm:convergence_brown_measure}{Theorem \ref{thm:convergence_brown_measure}}
\begin{thm:convergence_brown_measure}
	Consider the random matrix model $X_n = P_n + i Q_n$, where $P_n, Q_n \in M_n(\C)$ are independently Haar-rotated Hermitian matrices with distributions with at most $2$ atoms. Suppose that the law of $P_n$ converges to the law of $p$ and the law of $Q_n$ converges to the law of $q$, where $p, q \in (M, \tau)$ are freely independent. Then, the empirical spectral distribution of $X_n$ converges almost surely in the vague topology to the Brown measure of $X = p + i q$.  
\end{thm:convergence_brown_measure}

Note that if the law of $P_n$ converges to the law of $p$, then $\mu_p$ is indeed supported on at most $2$ points: if the support of $\mu_p$ contained more points, then by testing the convergence of non-negative, compactly supported, continuous functions that are $1$ on neighborhoods of these points, we get a contradiction to the law of $P_n$ having at most $2$ atoms. The analogous result holds if the law of $Q_n$ converges to $q$. 

The vast majority of our effort will be spent dealing with the situation when $p$ and $q$ have $2$ atoms. The alternative situation where either $p$ or $q$ is constant (and hence real numbers) are special cases that can be handled by analyzing limiting cases of the Brown measure and using properties of the Brown measure of $X = p + i q$. 

The following are figures generated using Mathematica illustrate the convergence, comparing the empirical spectral distributions of some deterministic $X_n = P_n + i Q_n$ for $n = 200, 750$:

\begin{figure}[H]
	\centering
	\begin{subfigure}{0.45 \textwidth}
		\centering
		\includegraphics[width = .9 \textwidth]{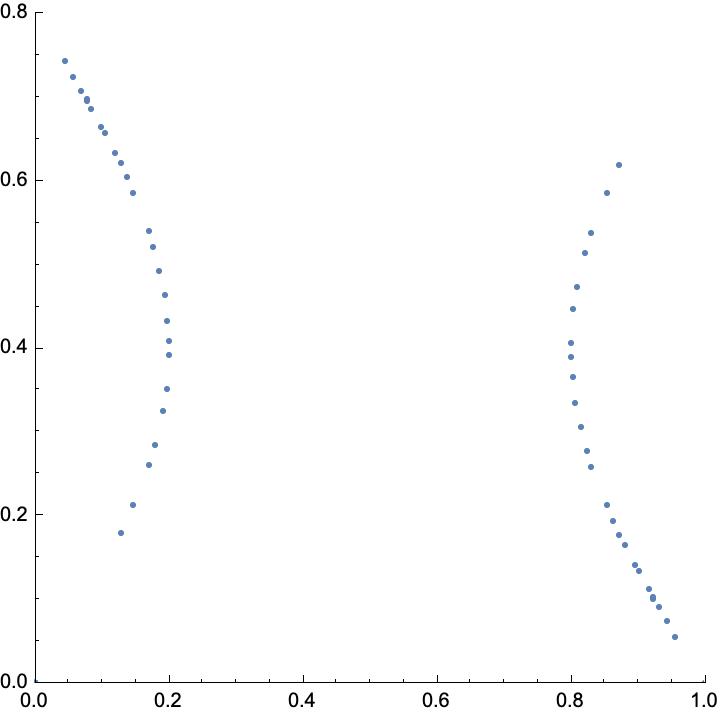}
		\caption{ESD of $X_n$, $n = 200$}
	\end{subfigure}
	\hfill
	\begin{subfigure}{0.45 \textwidth}
		\centering
		\includegraphics[width = .9 \textwidth]{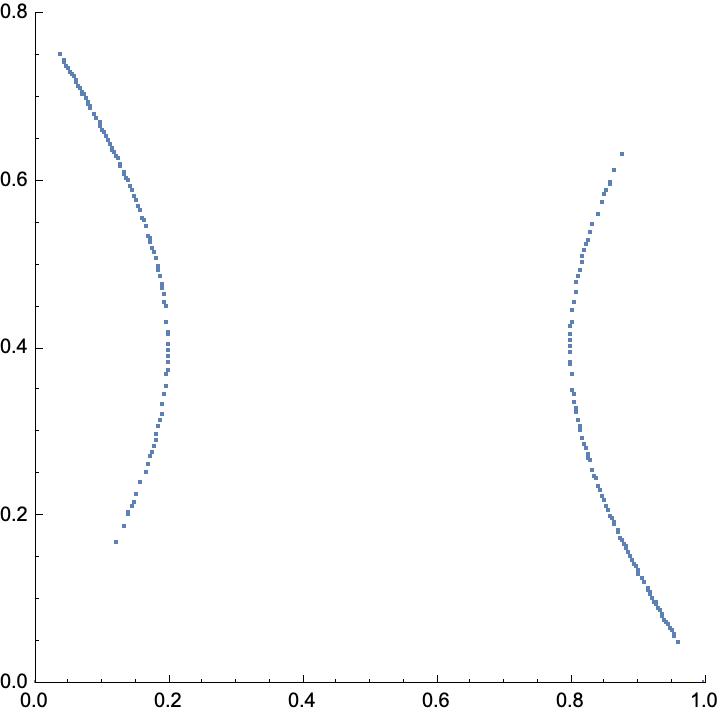}
		\caption{ESD of $X_n$, $n = 750$}
	\end{subfigure}ç
	\caption{$X_n = P_n + i Q_n$ \\ $\mu_{P_n} \approx (5/8) \delta_{0} + (3/8) \delta_{1}$ \\ $\mu_{Q_n} \approx (7/8) \delta_0 + (1/8) \delta_{4/5}$ }
	\label{fig:thm:boundary_curve}
\end{figure} 

We also prove a converse result in Theorem \ref{thm:converse} that the Brown measure of operators of the form $X = p + i q$ are the only possible limits for the empirical spectral distributions of $X_n$:

\newtheorem*{thm:converse}{Theorem \ref{thm:converse}}
\begin{thm:converse}
	Consider the random matrix model $X_n = P_n + i Q_n$, where $P_n, Q_n \in M_n(\C)$ are independently Haar-rotated Hermitian matrices with distributions with at most $2$ atoms. Suppose that the empirical spectral distribution of $X_n$ converges in probability in the vague topology to some deterministic probability measure $\mu$. Then, the law of $P_n$ converges to the law of $p$ and the law of $Q_n$ converges to the law of $q$ for some $p, q \in (M, \tau)$ where $p$ and $q$ are freely independent. Hence, $\mu$ is the Brown measure of $X = p + i q$.
\end{thm:converse}

The rest of the paper is organized as follows. In Section \ref{sec:preliminaries}, we recall the definition of the Brown measure, some properties of the Brown measure of $X = p + i q$ from \cite{BrownMeasurePaper1}, and outline the Hermitization method. In Section \ref{sec:spec_measures_converge}, we complete the first step of this method. In Section \ref{sec:min_singular_value}, we prove the minimum singular value estimates. In Section \ref{sec:convergence_converse}, we prove Theorems \ref{thm:convergence_brown_measure} and \ref{thm:converse}.

\section{Preliminaries}
\label{sec:preliminaries}

\subsection{Brown measure}
In this subsection, we recall the definition of the Brown measure of an operator $X \in (M, \tau)$ and the relevant properties. 

First, we recall the definition of the Fuglede-Kadison determinant, first introduced in \cite{KadisonPaper}: 

\begin{definition}
	Let $x \in (M, \tau)$. Let $\mu_{\Abs{x}}$ be the spectral measure of $\Abs{x} = (x^* x)^{1/2}$. Then, the \textbf{Fuglede-Kadison determinant} of $x$, $\Delta(x)$, is given by: 
	\begin{equation}
		\Delta(x) = \exp \left[ \int_{0}^{\infty} \log t \, d \mu_{\Abs{x}}(t) \right]   \,.
	\end{equation}
\end{definition}

When $x = X_n \in (M_n(\C), \frac{1}{n} \tr)$, then $\Delta(X_n) = \Abs{\det(X_n)}^{1/n}$. Thus, the Fuglede-Kadison determinant is a generalization of the normalized, positive determinant of a complex matrix. 

By applying the functional calculus to a decreasing sequence of continuous functions on $[0, \infty)$ converging to $\log t$, we have: 

\begin{equation}
	\log \Delta(x) = \int_{0}^{\infty} \log t \, d \mu_{\Abs{x}}(t) = \frac{1}{2} \int_{0}^{\infty} \log t \, d \mu_{\Abs{x}^2}(t) \,.
\end{equation}

In practice, it is more convenient to compute the right-hand side to compute $\Delta(X)$. 

Then, the Brown measure is the following distribution:

\begin{definition}
	Let $x \in (M, \tau)$. The \textbf{Brown measure} of $x$ is defined as:
	\begin{equation}
		\mu_x = \frac{1}{2 \pi}\nabla^2 \log \Delta(z - x) \,.
	\end{equation}
\end{definition} 

In \cite{BrownPaper}, some properties of the Brown measure are proven (see also \cite{SpeicherBook}, \cite{TaoBook}, \cite{HaagerupPaper} for more exposition on some of the basic results). In particular, the Brown measure of $x$ is a probability measure supported on the spectrum of $x$. When $x$ is normal, then the Brown measure is the spectral measure. When $x$ is a random matrix, then the Brown measure is the empirical spectral distribution.

Further, the Riesz representation for the subharmonic function $L(z) = \log \Delta(z - x)$ is given by: 
\begin{equation}
	\log \Delta(z - x) = \int_{\C}^{} \log \Abs{z - w} \, d \mu_x(w) \,,
\end{equation}
and the Brown measure is the unique complex Borel measure to satisfy this equation.

\subsection{Brown measure of $X = p + i q$}
In this subsection, we recall results from \cite{BrownMeasurePaper1} on the Brown measure of $X = p + i q$, where $p, q \in (M, \tau)$ are Hermitian, freely independent, and have spectra consisting of $2$ atoms. Specifically, let 

\begin{equation}
	\begin{aligned}
		\mu_p & = a \delta_\alpha + (1 - a) \delta_{\alpha'} \\
		\mu_q &= b \delta_\beta + (1 - b) \delta_{\beta'} \,,
	\end{aligned}
\end{equation}
where $a, b \in (0, 1)$, $\alpha, \alpha', \beta, \beta' \in \R$, $\alpha \neq \alpha'$, and $\beta \neq \beta'$.

The exact formulas describing the measure are not necessary, only some key properties which we restate. 

First, we use the definition of the hyperbola and rectangle associated with $X = p + i q$: 
\begin{definition}
	\label{def:hyperbola_rectangle}
	Let $\mathscr{A} = \alpha' - \alpha$ and $\mathscr{B} = \beta' - \beta$. 
	
	The \textbf{hyperbola associated with $X$} is 
	\begin{equation}
		H = \left\lbrace z = x + i y \in \C :	\left( x - \frac{\alpha + \alpha'}{2}  \right)^2 - \left(  y - \frac{\beta + \beta'}{2} \right)^2 = \frac{\mathscr{A}^2 - \mathscr{B}^2}{4}    \right\rbrace \,.
	\end{equation}
	The \textbf{rectangle associated with $X$} is
	\begin{equation}
		R =  \left\lbrace z = x + i y \in \C : x \in  [\alpha \wedge \alpha', \alpha \vee \alpha' ], y \in [\beta \wedge \beta', \beta \vee \beta'] \right\rbrace  \,.
	\end{equation}
\end{definition} 

The Brown measure of $X$ is the convex combination of 4 atoms with another measure $\mu'$: 
\begin{equation}
	\mu = \epsilon_{0 0} \delta_{ \alpha + i \beta } + \epsilon_{0 1} \delta_{\alpha + i \beta'} + \epsilon_{1 0} \delta_{\alpha'+ i \beta} + \epsilon_{1 1} \delta_{\alpha'+ i \beta'} + \epsilon \mu' \,,
\end{equation}

where 
\begin{equation}
	\begin{aligned}
		\epsilon_{0 0} & = \max(0, a  + b - 1) \\
		\epsilon_{0 1} & = \max(0, a + (1 - b) - 1) \\
		\epsilon_{1 0} & = \max(0, (1 - a) + b - 1  )\\
		\epsilon_{1 1} & = \max(0, (1 - a) + (1 - b) - 1) \\
		\epsilon &= 1 - (\epsilon_{0 0} + \epsilon_{0 1} + \epsilon_{1 0} + \epsilon_{1 1}  ) \,.
	\end{aligned}
\end{equation}	
and $\mu'$ is a measure supported on $H \cap R$. Note that the atoms are at the corners of $H \cap R$, so the entire Brown measure $\mu$ is supported on $H \cap R$.

A useful Corollary is the following, which applies even if $p$ or $q$ is constant: 

\begin{corollary}
	\label{cor:brown_measure_injective}
	Let $p, q \in (M, \tau)$ be Hermitian and freely independent and
	\begin{equation}
		\begin{aligned}
			\mu_p & = a \delta_\alpha + (1 - a) \delta_{\alpha'} \\
			\mu_q &= b \delta_\beta + (1 - b) \delta_{\beta'} \,.
		\end{aligned}
	\end{equation}
	where $a, b \in [0, 1]$ and $\alpha, \alpha', \beta, \beta' \in \R$. Let $\mu$ be the Brown measure of $X = p + i q$. Then, the assignment $(\mu_p, \mu_q) \mapsto \mu$ is 1 to 1. 
\end{corollary}

A useful Lemma from \cite[Lemma 4.4]{BrownMeasurePaper1} about the hyperbola and rectangle associated with $X$ is: 

\begin{lemma}
	\label{lem:hyperbola_rectangle}
	Let $\alpha, \alpha', \beta, \beta' \in \R$, where $\alpha \neq \alpha'$ and $\beta \neq \beta'$. Let $\mathscr{A} = \alpha' - \alpha$ and $\mathscr{B} = \beta' - \beta$.
	
	Let 
	\begin{equation}
		\begin{aligned}
			H & = \left\lbrace z = x + i y \in \C :	\left( x - \frac{\alpha + \alpha'}{2}  \right)^2 - \left(  y - \frac{\beta + \beta'}{2} \right)^2 = \frac{\mathscr{A}^2 - \mathscr{B}^2}{4}    \right\rbrace \\
			R & =  \left\lbrace z = x + i y \in \C : x \in  [\alpha \wedge \alpha', \alpha \vee \alpha' ], y \in [\beta \wedge \beta', \beta \vee \beta'] \right\rbrace \,.
		\end{aligned}
	\end{equation}		
	The equation of $H$ is equivalent to: 
	\begin{equation}
		(x - \alpha)(x - \alpha') = (y - \beta)(y - \beta') \,.
	\end{equation}
	The equation of $H$ in coordinates 
	\begin{equation}
		\begin{aligned}
			x' & = x - \frac{\alpha + \alpha'}{2} \\
			y' &= y - \frac{\beta + \beta'}{2}
		\end{aligned}
	\end{equation}
	is 
	\begin{equation}
		\label{eqn:lem:hyperbola_rectangle}
		(x')^2 - \frac{\mathscr{A}^2  }{4} = (y')^2 - \frac{\mathscr{B}^2}{4} .
	\end{equation}
	It follows that for $(x, y) \in H$, 
	\begin{equation}
		(x, y) \in R \iff (\ref{eqn:lem:hyperbola_rectangle}) \leq 0 \iff x \in  [\alpha \wedge \alpha', \alpha \vee \alpha' ] \text{ or } y \in [\beta \wedge \beta', \beta \vee \beta'] \,.
	\end{equation}
	Similarly, 
	\begin{equation}
		(x, y) \in \interior{R} \iff (\ref{eqn:lem:hyperbola_rectangle}) < 0 \iff x \in  (\alpha \wedge \alpha', \alpha \vee \alpha') \text{ or } y \in (\beta \wedge \beta', \beta \vee \beta') \,.
	\end{equation}	
	Alternatively, the equation of the hyperbola is:
	\begin{equation}
		\Re \left( \left( z - \frac{\alpha + \alpha'}{2} - i \frac{\beta + \beta'}{2}  \right)^2 \right) =   \frac{\mathscr{A}^2 - \mathscr{B}^2}{4} \,.
	\end{equation}
	If $z \in H$, then $z \in R$ if and only if
	\begin{equation}
		\Abs{\Im \left( \left( z - \frac{\alpha + \alpha'}{2} - i \frac{\beta + \beta'}{2}  \right)^2\right)}  \leq \frac{\Abs{\mathscr{A} \mathscr{B}}}{2} \,.
	\end{equation}
\end{lemma}
\begin{proof}
	The equivalent equations for the hyperbola are straightforward to check. 
	
	The equivalences for the closed conditions follow from the following equivalences and the equation of the hyperbola in $x', y'$ coordinates
	\begin{equation}
		\begin{aligned}
			(x, y) \in R & \iff (x')^2 - \frac{\mathscr{A}^2}{4} \leq 0 \text{ and } (y')^2 - \frac{\mathscr{B}^2}{4} \leq 0 \\
		\end{aligned}
	\end{equation}
	\begin{equation}
		\begin{aligned}
			(x')^2 - \frac{\mathscr{A}^2  }{4} \leq 0  & \iff \Abs{x'} \leq \frac{\Abs{\mathscr{A}}}{2} & \iff x \in [\alpha \wedge \alpha', \alpha \vee \alpha' ]  \\
			(y')^2 - \frac{\mathscr{B}^2}{4} \leq 0 & \iff  \Abs{y'} \leq \frac{\Abs{\mathscr{B}}}{2} & \iff y \in [\beta \wedge \beta', \beta \vee \beta']  \,.
		\end{aligned}
	\end{equation}
	The equivalences for the open conditions follow from similar equivalences with the closed conditions replaced by open conditions.
	
	The last equation of the hyperbola follows from direct computation. For the inequality of the rectangle, observe that
	\begin{equation}
		\Im \left( \left( z - \frac{\alpha + \alpha'}{2} - i \frac{\beta + \beta'}{2}  \right)^2  \right)  = 2 x' y' \,.
	\end{equation} 
	In light of what was previously shown, 
	\begin{equation}
		x' y' \leq \frac{\Abs{\mathscr{A} \mathscr{B}}}{4} \Longrightarrow x' \leq \frac{\Abs{\mathscr{A}}}{2} \text{ or } y' \leq \frac{\Abs{\mathscr{B}}}{2} \Longrightarrow z \in R \,.
	\end{equation}
	Conversely, 
	\begin{equation}
		z \in R \Longrightarrow x' \leq \frac{\Abs{\mathscr{A}}}{2} \text{ and } y' \leq \frac{\Abs{\mathscr{B}}}{2} \Longrightarrow x' y' \leq \frac{\Abs{\mathscr{A} \mathscr{B}}}{4} \,.
	\end{equation}
\end{proof}

\subsection{Hermitization method}
In this subsection, we review the Hermitization method and give an outline of the steps of the proof of Theorem \ref{thm:convergence_brown_measure}. We refer to \cite{TaoBook} and \cite{GuionnetPaper} for the justification of facts stated without proof. 

For the Hermitization method, we consider the logarithmic potential of compactly supported, positive measure on $\C$: 

\begin{definition}
	Let $\mu$ be a compactly supported, positive Borel measure on $\C$. Then, the \textbf{logarithmic potential} of $\mu$, $L_\mu: \C \to \C$, is given by: 
\begin{equation}
	L_\mu(z) = \int_{\C}^{} \log \Abs{z - w} \, d \mu(w) \,.
\end{equation}
\end{definition}
Hence, if $\mu$ is the Brown measure of $x \in (M, \tau)$, then 
\begin{equation}
	L_\mu(z) = \log \Delta(z - x)  = \int_{0}^{\infty} \log t \, d \mu_{\Abs{z - x}}(t) = \frac{1}{2} \int_{0}^{\infty} \log t \, d \mu_{\Abs{z - x}^2}(t)\,.
\end{equation}
In particular, this identity holds when $x = X_n$ is a random matrix, where the Brown measure $\mu = \mu_n$ is the empirical spectral distribution of $X_n$.

Since $\mu$ is compactly supported, then $L_\mu \in L^{1}_\text{loc}(\C)$ (with respect to the Lebesgue measure). When $\mu = \mu_n$ is the empirical spectral distribution of a random matrix, this follows from the fact that $\log \Abs{\cdot - w} \in L^1_\text{loc}(\C)$ for any $w \in \C$. In general, $L_\mu \in L^1_\text{loc}(\C)$ follows from the fact that $\log \Abs{\cdot - w} \in L^1_\text{loc}(\C)$ uniformly for $w$ in a compact set. 

In particular, $L_\mu$ is a well-defined distribution and since $\frac{1}{2\pi} \nabla^2 \log \Abs{\cdot - w} = \delta_{w}$, then
\begin{equation}
	\frac{1}{2 \pi} \nabla^2  L_\mu = \mu \, ,
\end{equation}
so we can recover $\mu$ from $L_\mu$.  

In the context of Theorem \ref{thm:convergence_brown_measure}, let $\mu_n$ be the empirical spectral distribution of the random matrices $X_n$ and let $\mu$ be the Brown measure of $X = p + i q$. Then, if we can prove the limit: 
\begin{equation}
	L_{\mu_n}(z) \to L_\mu(z) \, ,
\end{equation}
then we should be able to take distributional Laplacians and conclude that $\mu_n \to \mu$. 

The following Proposition \cite[Theorem 2.8.3]{TaoBook} makes this precise: 
\begin{proposition}
	\label{prop:log_potential_convergence}
	Let $\mu_n$ be a sequence of random probability measures on $\C$ and suppose that $\mu$ is a deterministic probability measure on $\C$. Assume that $\mu_n, \mu$ are almost surely supported on some compact set. Suppose for Lebesgue almost every $z \in \C$, $L_{\mu_n}(z) \to L_\mu(z)$ almost surely (resp. in probability). Then, $\mu_n$ converges almost surely (resp. in probability) in the vague topology. 
\end{proposition}

As in our case, the boundedness hypothesis in Proposition \ref{prop:log_potential_convergence} is typically not an issue. 

We highlight some notation that we will use for the rest of the paper: 

\begin{definition}
	Let $X_n$ be the random matrix model from Definition \ref{def:X_n}, and suppose that the law of $P_n$ converges to the law of $p$ and the law of $Q_n$ converges to the law of $q$, where $p, q \in (M, \tau)$ are freely independent. Then, define the following quantities: 
	\begin{itemize}
		\item Let $\bm{H_z(X_n)} = (z - X_n)^*(z - X_n)$.
		\item Let $\bm{\nu_{n, z}}$ be the spectral measure of $H_z(X_n)$.
		\item Let $\bm{H_z(X)} = (z - X)^*(z - X)$.
		\item Let $\bm{\nu_z}$ be the spectral measure of $H_z(X)$.
		\item Let $\bm{\mu_n}$ be empirical spectral distribution of $X_n$.
		\item Let $\bm{\mu}$ be the Brown measure of $X = p + i q$.
	\end{itemize}
\end{definition}

Thus, to prove Theorem \ref{thm:convergence_brown_measure}, it suffices to prove the convergence of the following integrals for almost every $z \in \C$ almost surely: 
\begin{equation}
	\frac{1}{2} \int_{0}^{\infty} \log x \, d \nu_{n, z}(x) \to \frac{1}{2} \int_{0}^{\infty} \log x \, d \nu_{z}(x) \,.
\end{equation}
In general, the strength of the convergence of these integrals (i.e. whether the convergence is in probability or almost surely) determines the strength of the convergence of $\mu_n$ to $\mu$.

The main difficulty in justifying the convergence of the logarithmic integrals is the fact that $\log x$ is not continuous at $0$: $\log x \to - \infty$ as $x \to 0^-$. So, the logarithmic integrals may not converge if the measures $\nu_{n, z}$ have too much mass near $0$. Note that the measure $\nu_{n, z}$ is supported on $[\sigma_{\min}(z - X_n)^2, \norm{z - X_n}]$, so in order to prove the convergence of the logarithmic integrals, it suffices bound the minimum singular value from below. 

In practice, this is the most difficult part of these convergence arguments. For our $X_n$, we use the algebraic properties of projections and the geometric properties of the support of the $X_n$.

To summarize, in order to prove Theorem \ref{thm:convergence_brown_measure}, we need to complete the following steps:

\begin{enumerate}
	\item Show that for Lebesgue almost every $z \in \C$, the spectral measures $\nu_{n, z}$ converge to $\nu_z$ almost surely in the vague topology.
	\item Bound the minimum singular value of $z - X_n$ from below to justify the convergence
	\begin{equation}
		\frac{1}{2} \int_{0}^{\infty} \log x \, d \nu_{n, z}(x) \to \frac{1}{2} \int_{0}^{\infty} \log x \, d \nu_{z}(x) 
	\end{equation}
	for Lebesgue almost every $z \in \C$ almost surely. 
\end{enumerate}

We will prove these steps exclusively in the case where $p$ and $q$ have $2$ atoms (resp. $P_n$ and $Q_n$ have $2$ atoms), even if the results can be generalized for the case when $p$ or $q$ is constant (resp. $P_n$ or $Q_n$ is constant). The exceptional cases will be handled in the final proof of Theorem \ref{thm:convergence_brown_measure}.

In the case $p$ and $q$ have spectra consisting of $2$ atoms, we use the following notation for their spectral measures:

\begin{definition}
	\label{def:p_q_spectra}
	If $p$ and $q$ have spectra consisting of $2$ atoms, let
	\begin{equation}
		\begin{aligned}
			\bm{\mu_p} & = a \delta_\alpha + (1 - a) \delta_{\alpha'} \\
			\bm{\mu_q} &= b \delta_\beta + (1 - b) \delta_{\beta'} \,,
		\end{aligned}
	\end{equation}
	where $a, b \in (0, 1)$, $\alpha, \alpha', \beta, \beta' \in \R$, $\alpha \neq \alpha'$, and $\beta \neq \beta'$.
\end{definition}

In this situation, if the law of $P_n$ converges to the law of $p$ and the law of $Q_n$ converges to the law of $q$, then the weights and positions of the atoms of $P_n$ (resp. $Q_n$) converge to the weights and positions of the atoms of $p$ (resp. $q$): by testing the convergence $\mu_{P_n} \to \mu_p$ on non-negative $f \in C_c(\R)$ that is supported on a neighborhood of $\alpha$ and $f(\alpha) = 1$, then an atom of $\mu_{P_n}$ converges to $\alpha$. A similar argument shows that an atom of $\mu_{P_n}$ converges to $\alpha'$. By choosing $f$ to be $1$ on neighborhoods of $\alpha$ and $\alpha'$, we see that the weights of the atoms of $P_n$ converge to the weights of the corresponding atoms of $p$. Thus, we may assume that $\alpha_n \to \alpha$, $\beta_n \to \beta$, $\alpha_n' \to \alpha'$, $\beta_n' \to \beta'$, $a_n \to a$, and $b_n \to b$.

\section{Convergence of Spectral Measures}
\label{sec:spec_measures_converge}
In this section, we complete the first step of the outline of the proof of Theorem \ref{thm:convergence_brown_measure}, proving the convergence of $\nu_{n, z}$ to $\nu_z$ for almost every $z \in \C$, almost surely in the vague topology. We do this in the case where $p$ and $q$ have spectra consisting of $2$ atoms, and use the notation of Definition \ref{def:p_q_spectra}. Recall that in this situation, $\alpha_n \to \alpha$, $\beta_n \to \beta$, $\alpha_n' \to \alpha'$, $\beta_n' \to \beta'$, $a_n \to a$, and $b_n \to b$.

By making a choice of commuting $P_n', Q_n'$ in Definition \ref{def:X_n}, we deduce the following result from the asymptotic freeness of independent Haar unitaries:

\begin{proposition}
	\label{prop:projections_free}
	The joint law of $P_n, Q_n$ in $\left( M_n(\C), \mathbb{E}\left[ \frac{1}{n} \tr   \right]  \right)$ \\(resp. $ \left(M_n(\C), \frac{1}{n} \tr \right) $) is asymptotically free (resp. almost surely asymptotically free), converging to the law of two free Hermitian operators $p, q \in (M, \tau)$ where 
	\begin{equation}
		\begin{aligned}
			\mu_p & = a \delta_\alpha + (1 - a) \delta_{\alpha'} \\
			\mu_q &= b \delta_\beta + (1 - b) \delta_{\beta'} \,.
		\end{aligned}
	\end{equation}	
\end{proposition}
\begin{proof}
	Recall that $P_n = U_n P_n' U_n^*$, $Q_n = V_n Q_n' V_n^*$, where $U_n, V_n$ are independent Haar-distributed unitaries and $P_n', Q_n'$ are deterministic, Hermitian, and 
	\begin{equation}
		\begin{aligned}
			\mu_{P_n'} & = a_n \delta_{\alpha_n} + (1 - a_n) \delta_{\alpha_n'} \\
			\mu_{Q_n'} &= b_n \delta_{\beta_n} + (1 - b_n) \delta_{\beta_n'} \,.
		\end{aligned}
	\end{equation}
	Because of the invariance property of the Haar measure, any choice of the $P_n'$ (resp. $Q_n'$) gives the same law for $P$ (resp. $Q$). Thus, let $P_n', Q_n'$ be given by: 
	\begin{equation}
		\begin{aligned}
			P_n' & = (\alpha_n' - \alpha_n) \tilde{P_n} + \alpha_n \\
			Q_n' &= (\beta_n'- \beta_n) \tilde{Q_n} + \beta_n \,,
		\end{aligned}
	\end{equation}
	where $\tilde{P_n}$ is the matrix of the projection onto the first $n (1 - a_n)$ standard basis vectors of $\C^n$ and $\tilde{Q_n}$ is the matrix of the projection onto the first $n (1 - b_n)$ standard basis vectors of $\C^n$. 
	
	Then, $\tau(\tilde{P_n}) = 1 - a_n$, $\tau(\tilde{Q_n}) = 1 - b_n$ and 
	\begin{equation}
		\begin{aligned}
			\tilde{P_n} & = \tilde{P_n} \wedge \tilde{Q_n} + \tilde{P_n} \wedge (1 - \tilde{Q_n}) \\
			\tilde{Q_n} &= \tilde{P_n} \wedge \tilde{Q_n} + (1 - \tilde{P_n}) \wedge \tilde{Q_n}
		\end{aligned}
	\end{equation}
	\begin{equation}
		\begin{aligned}
			\tau(\tilde{P_n} \wedge \tilde{Q_n}) & = \min(\tau(\tilde{P_n}), \tau(\tilde{Q_n})) \\
			\tau(\tilde{P_n} \wedge (1 - \tilde{Q_n})) &= \tau(\tilde{P_n}) - \tau(\tilde{P_n} \wedge \tilde{Q_n}) \\
			\tau( (1 - \tilde{P_n}) \wedge \tilde{Q_n} ) &= \tau(\tilde{Q_n}) - \tau(\tilde{P_n} \wedge \tilde{Q_n})  \\
			\tau( (1 - \tilde{P_n}) \wedge (1 - \tilde{Q_n}) ) &= 1 - \max(\tau(\tilde{P_n}), \tau(\tilde{Q_n})) \,.
		\end{aligned}
	\end{equation}
	The projections $\tilde{P_n} \wedge \tilde{Q_n},  \tilde{P_n} \wedge (1 - \tilde{Q_n}), (1 - \tilde{P_n}) \wedge \tilde{Q_n}, (1 - \tilde{P_n}) \wedge (1 - \tilde{Q_n})$ are mutually orthogonal and sum to $1$. Hence, their joint law converges to the joint law of mutually orthogonal projections $e_{1 1}, e_{1 0}, e_{0 1}, e_{00}$ that sum to $1$, where
	\begin{equation}
		\begin{aligned}
			\tau(e_{11}) & = \min( 1 - a, 1 - b ) \\
			\tau(e_{1 0}) & = 1 - a - \min( 1 - a, 1 - b ) \\
			\tau(e_{0 1}) & = 1 - b - \min( 1 - a, 1 - b ) \\
			\tau(e_{0 0}) & = 1 - \max(1 - a, 1 - b) \,.
		\end{aligned}
	\end{equation}
	By forming the variables: 
	\begin{equation}
		\begin{aligned}
			\tilde{p} & = e_{1 1} + e_{1 0} \\
			\tilde{q} &= e_{1 1} + e_{0 1} \\
			p' &=  (\alpha' - \alpha)\tilde{p} + \alpha \\
			q' &= (\beta' - \beta) \tilde{q} + \beta \,.
		\end{aligned}
	\end{equation}
	then it follows that the law of $P_n', Q_n'$ converges to the law of $p', q'$, where 
	\begin{equation}
		\begin{aligned}
			\mu_{p'} & = a \delta_\alpha + (1 - a) \delta_{\alpha'} \\
			\mu_{q'} &= b \delta_\beta + (1 - b) \delta_{\beta'} \,.
		\end{aligned}
	\end{equation}	
	Since $P_n', Q_n'$ are deterministic, then they are independent from $U_n, V_n$. Further, $\norm{P_n'} = \max(\alpha_n, \alpha_n')$ and $\norm{Q_n'} = \max(\beta_n, \beta_n')$ and since these sequences converge, then $\norm{P_n'}$, $\norm{Q_n'}$ are uniformly bounded. Hence, from the asymptotic freeness of independent Haar unitaries (see \cite[Theorem 5.4.10]{GuionnetBook}), the law of $P_n', Q_n'$ becomes (almost surely) asymptotically free from $U_n, V_n$. Let $u, v$ be the limit variables for $U_n, V_n$. It is easy to check that $p = u p' u^*$ and $v q' v^*$ are freely independent. Hence, the joint law of $P_n, Q_n$ converges to the law of $p, q$, where $p$ and $q$ are free and
	\begin{equation}
		\begin{aligned}
			\mu_p & = a \delta_\alpha + (1 - a) \delta_{\alpha'} \\
			\mu_q &= b \delta_\beta + (1 - b) \delta_{\beta'} \,.
		\end{aligned}
	\end{equation}
\end{proof}

As a corollary, we observe that the laws $\nu_{n, z}$ converge to $\nu_z$:

\begin{corollary}
	\label{cor:nu_z_convergence}
	Let $\nu_{n, z}$ be the spectral measure of $H_z(X_n) = (z - X_n)^*(z - X_n)$ and let $\nu_z$ be the spectral measure of $H_z(X) = (z - X)^* (z - X)$. For every $z \in \C$, $\nu_{n, z}$ converges to $\nu_z$ almost surely in the vague topology.
\end{corollary}
\begin{proof}
	This follows immediately from Proposition \ref{prop:projections_free} by noting that the convergence of the laws of $H_z(X_n)$ to $H_z(X)$ (i.e. convergence of all moments) almost surely implies vague convergence of the $\mu_n$ to $\mu$ almost surely. 
\end{proof}

\section{Bounds on the Minimum Singular Value}
\label{sec:min_singular_value}
In this section, we complete the second major step of the proof of Theorem \ref{thm:convergence_brown_measure} in the case when $p$ and $q$ have $2$ atoms. Recall that this implies that for $n$ sufficiently large, $P_n$ and $Q_n$ also have $2$ atoms. We will assume this is true for all $n$, so that in Definition \ref{def:X_n}, $a_n, b_n \in (0, 1)$, $\alpha_n \neq \alpha_n'$, and $\beta_n \neq \beta_n'$. 

We will use the following notation for the rest of the paper:

\begin{definition}
	For $X_n$ as in Definition \ref{def:X_n}, let 
	\begin{equation}
		\begin{aligned}
			\tilde{\bm{P_n}} &= P_n - \frac{\alpha_n + \alpha_n'}{2} \\
			\tilde{\bm{Q_n}} &= Q_n - \frac{\beta_n + \beta_n'}{2} \\
			\tilde{\bm{X_n}} & = \tilde{P_n} + i \tilde{Q_n} \\
			\bm{\mathscr{A}_n} & = \alpha_n' - \alpha_n \\
			\bm{\mathscr{B}_n} &= \beta_n' - \beta_n \,.
		\end{aligned}	
	\end{equation}
\end{definition}

In this section, we will bound the minimum singular value of $z - X_n$ from below, for Lebesgue almost every $z \in \C$.

First, we understand the geometry of the support of $\mu_n$. Analogous to $H$ and $R$ from Definition \ref{def:hyperbola_rectangle}, let $H_n$ and $R_n$ be the hyperbola and rectangle associated with $X_n$: 

\begin{definition}
	\label{def:hyperbola_rectangle_n}
	Let $P_n, Q_n \in M_n(\C)$ be Hermitian with laws: 
	\begin{equation}
		\begin{aligned}
			\mu_{P_n} & = a_n \delta_{\alpha_n} + (1 - a_n) \delta_{\alpha_n'} \\
			\mu_{Q_n} &= b_n \delta_{\beta_n} + (1 - b_n) \delta_{\beta_n'} \,,
		\end{aligned}
	\end{equation}
	where $a_n, b_n \in (0, 1)$, $\alpha_n, \alpha_n', \beta_n, \beta_n' \in \R$, $\alpha_n \neq \alpha_n'$, and $\beta_n \neq \beta_n'$.
	
	Let $\mathscr{A}_n = \alpha_n' - \alpha_n$ and $\mathscr{B}_n = \beta_n' - \beta_n$. 
	
	The \textbf{hyperbola associated with $X_n$} is:
	\begin{equation}
		H_n = \left\lbrace z = x + i y \in \C :	\left( x - \frac{\alpha_n + \alpha_n'}{2}  \right)^2 - \left(  y - \frac{\beta_n + \beta_n'}{2} \right)^2 = \frac{  \mathscr{A}_n^2 - \mathscr{B}_n^2}{4}    \right\rbrace .
	\end{equation}
	The \textbf{rectangle associated with $X_n$} is:
	\begin{equation}
		R_n =  \left\lbrace z = x + i y \in \C : x \in  [\alpha_n \wedge \alpha_n', \alpha_n \vee \alpha_n' ], y \in [\beta_n \wedge \beta_n', \beta_n \vee \beta_n'] \right\rbrace  .
	\end{equation}
\end{definition}

We will show that $\mu_n$ is supported on $H_n \cap R_n$. First, we consider $\tilde{X_n}^2$:

\begin{proposition}
	\label{prop:tilde_X_n^2}
	$\tilde{X_n}^2$ is normal and 
	\begin{equation}
		\begin{aligned}
			\Re(\tilde{X_n}^2) & = \frac{\mathscr{A}_n^2 - \mathscr{B}_n^2}{4} \\
			\norm{\Im(\tilde{X_n}^2)} & \leq \frac{\Abs{\mathscr{A}_n \mathscr{B}_n}}{2} \,.
		\end{aligned}
	\end{equation}
	If $\rho$ is an eigenvalue of $\tilde{X_n}^2$ then 
	\begin{equation}
		\begin{aligned}
			\Re(\rho) &= \frac{\mathscr{A}_n^2 - \mathscr{B}_n^2}{4} \\
			\Abs{\Im(\rho)} &\leq \frac{\Abs{\mathscr{A}_n \mathscr{B}_n}}{2} \,.
		\end{aligned}
	\end{equation}
\end{proposition}
\begin{proof}
	Observe that $\tilde{P_n}^2 = \mathscr{A}_n^2 / 4$ and $\tilde{Q_n}^2 = \mathscr{B}_n^2 / 4$.
	
	Hence, 
	\begin{equation}
		\begin{aligned}
			\tilde{X_n}^2 & = (\tilde{P_n} + i \tilde{Q_n})^2
			\\ &= (\tilde{P_n}^2 -  \tilde{Q_n}^2 ) + i (\tilde{P_n} \tilde{Q_n} + \tilde{Q_n} \tilde{P_n})
			\\ &= \frac{\mathscr{A}_n^2 - \mathscr{B}_n^2}{4} + i (\tilde{P_n} \tilde{Q_n} +  \tilde{Q_n} \tilde{P_n}) \,.
		\end{aligned}  
	\end{equation}
	As $\tilde{P_n} \tilde{Q_n} + \tilde{P_n} \tilde{Q_n}$ is Hermitian then 
	\begin{equation}
		\begin{aligned}
			\Re(\tilde{X_n}^2) & = \frac{\mathscr{A}_n^2 - \mathscr{B}_n^2}{4} \\
			\Im(\tilde{X_n}^2) &= \tilde{P_n} \tilde{Q_n} + \tilde{Q_n} \tilde{P_n} \,.
		\end{aligned}
	\end{equation}	
	Note that from the spectra of $\tilde{P_n}, \tilde{Q_n}$ that $\norm{\tilde{P_n}} = \Abs{\mathscr{A}_n} / 2$ and $\norm{\tilde{Q_n}} = \Abs{\mathscr{B}_n} / 2$. Hence, 
	\begin{equation}
		\begin{aligned}
			\norm{\Im(\tilde{X_n}^2)} &= \norm{\tilde{P_n} \tilde{Q_n} + \tilde{Q_n} \tilde{P_n}} \\
			& \leq \norm{\tilde{P_n}} \norm{\tilde{Q_n}} + \norm{\tilde{Q_n}} \norm{\tilde{P_n}} \\
			& = \frac{\Abs{\mathscr{A}_n \mathscr{B}_n}}{2} \,.
		\end{aligned}
	\end{equation}
	Since $\Re(\tilde{X_n})$ is a constant, then clearly it commutes with $\Im(\tilde{X_n})$ and it follows that $\tilde{X_n}$ is normal. Hence, $\tilde{X_n}$ is diagonalizable and its eigenvalues are of the form $\rho = \rho_1 + i \rho_2$, where $\rho_1$ is an eigenvalue of $\Re(\tilde{X_n})$ and $\rho_2$ is an eigenvalue of $\Im(\tilde{X_n})$. Thus,
	\begin{equation}
		\begin{aligned}
			\Re(\rho)  &= \frac{\mathscr{A}_n^2 - \mathscr{B}_n^2}{4} \\
			\Abs{\Im(\rho)} & \leq \norm{\Im(\tilde{X_n})} \leq \frac{\Abs{\mathscr{A}_n \mathscr{B}_n}}{2} \,.
		\end{aligned}
	\end{equation}
\end{proof}

We conclude that the eigenvalues of $X_n = P_n + i Q_n$ lie on $H_n \cap R_n$.

\begin{proposition}
	\label{prop:x_eigenvalue}
	The eigenvalues of $X_n$ lie on $H_n \cap R_n$.
\end{proposition}
\begin{proof}
	If $\lambda$ is an eigenvalue for $X_n = P_n + i Q_n$, then
	\begin{equation}
		\left( \lambda - \frac{\alpha_n + \alpha_n'}{2} - i \frac{\beta_n + \beta_n'}{2}  \right)^2  .
	\end{equation}
	is an eigenvalue for $\tilde{X_n}^2$. Hence, from Proposition \ref{prop:tilde_X_n^2}, 
	\begin{equation}
		\begin{aligned}
			\Re \left(  \left( \lambda - \frac{\alpha_n + \alpha_n'}{2} - i \frac{\beta_n + \beta_n'}{2}  \right)^2 \right)  &= \frac{\mathscr{A}_n^2 - \mathscr{B}_n^2}{4} \\
			\Abs{\Im \left(  \left( \lambda - \frac{\alpha_n + \alpha_n'}{2} - i \frac{\beta_n + \beta_n'}{2}  \right)^2 \right)} & \leq \frac{\Abs{\mathscr{A}_n \mathscr{B}_n}}{2} .
		\end{aligned}
	\end{equation}
	From Lemma \ref{lem:hyperbola_rectangle}, $\lambda \in H_n \cap R_n$.
\end{proof}

We introduce the following notation for the eigenspaces of a (random) matrix $Y_n \in M_n(\C)$: 

\begin{definition}
	\label{def:eigenspace}
	Let $Y_n \in M_n(\C)$. Define the following: 
	\begin{itemize}
		\item Let $E_\lambda(Y_n)$ be the $\mathbf{\lambda}$\textbf{-eigenspace} for $Y_n$, i.e. $E_\lambda(Y_n) = \ker(Y_n - \lambda I_n)$.
		\item Let $V_\lambda(Y_n)$ be the \textbf{generalized }$\mathbf{\lambda}$\textbf{-eigenspace} for $Y_n$, i.e. $E_\lambda(Y_n) = \bigcup_{k \geq 0} \ker( (Y_n - \lambda I_n)^k ) = \ker( (Y_n - \lambda I_n)^n )$.
	\end{itemize}
	In general, we allow for $E_\lambda(Y_n) = \{0\}$ or $V_\lambda(Y_n) = \{0\}$, but if we specifically say that $\lambda$ is an \textbf{eigenvalue}, then it is implied that $E_\lambda(Y_n) \neq \{0\}$ (and hence $V_\lambda(Y_n) \neq \{0\}$).	
\end{definition}

Next, we examine the relationship between the (generalized) eigenspaces of $X_n$ and $\tilde{X_n}^2$ to see that $X_n$ is almost diagonalizable:

\begin{proposition}	
	\label{prop:x_eigenspaces}
	Let $\rho \in \C$.
	If $\rho \neq 0$, let 
	\begin{equation}
		\begin{aligned}
			\lambda_{+}(\rho) & = \frac{\alpha_n + \alpha_n'}{2} + i \frac{\beta_n + \beta_n'}{2} + \sqrt{\rho} \\
			\lambda_{-}(\rho) & = \frac{\alpha_n + \alpha_n'}{2} + i \frac{\beta_n + \beta_n'}{2} - \sqrt{\rho} \,,
		\end{aligned}
	\end{equation}
	where $\sqrt{\rho}$ is a chosen square root of $\rho$. 
	Then, 
	\begin{equation}
		E_\rho(\tilde{X_n}^2) = 
		\begin{dcases}
			E_{\lambda_{+}(\rho)}(X_n) + E_{\lambda_{-}(\rho)}(X_n) & \rho \neq 0 \\
			V_{\frac{\alpha_n + \alpha_n'}{2} + i \frac{\beta_n + \beta_n'}{2}}(X_n) & \rho = 0 \,.
		\end{dcases}
	\end{equation}
	Further,
	\begin{equation}
		\begin{aligned}
			V_{\lambda_+(\rho)}(X_n) & = E_{\lambda_+(\rho)}(X_n) \\
			V_{\lambda_-(\rho)}(X_n) & = E_{\lambda_-(\rho)}(X_n) \\
			V_{\frac{\alpha_n + \alpha_n'}{2} + i \frac{\beta_n + \beta_n'}{2}}(X_n) & = \ker \left(  \left( X_n - \frac{\alpha_n + \alpha_n'}{2} - i \frac{\beta_n + \beta_n'}{2}\right)^2    \right)  \,.
		\end{aligned}
	\end{equation}	
\end{proposition}
\begin{proof}
	Since $\tilde{X_n}^2$ and $X_n$ commute, then $X_n$ fixes the $E_\rho(\tilde{X_n}^2)$. 
	
	Consider $\restr{X_n}{E_\rho(\tilde{X_n}^2)}: E_\rho(\tilde{X_n}^2) \to E_\rho(\tilde{X_n}^2) $. Then, $\restr{X_n}{E_\rho(\tilde{X_n}^2)}$ satisfies the polynomial
	\begin{equation}
		p(x) = \left( x - \frac{\alpha_n + \alpha_n'}{2} - i \frac{\beta_n + \beta_n'}{2} \right)^2 - \rho \,.
	\end{equation} 
	When $\rho \neq 0$, $p(x)$ is separable with roots $\lambda_{+}(\rho), \lambda_{-}(\rho)$. Hence, 
	\begin{equation}
		E_\rho(\tilde{X_n}^2) = E_{\lambda_{+}(\rho)} \left( \restr{X_n}{E_\rho(\tilde{X_n}^2)}\right)  + E_{\lambda_{-}(\rho)}\left( \restr{X_n}{E_\rho(\tilde{X_n}^2)}\right) \,.
	\end{equation}
	For $\rho = 0$, 
	\begin{equation}
		0 = p \left( \restr{X_n}{E_\rho(\tilde{X_n}^2)}\right) = \left( \restr{X_n}{E_\rho(\tilde{X_n}^2)} - \frac{\alpha_n + \alpha_n'}{2} - i \frac{\beta_n + \beta_n'}{2} \right)^2 \,.
	\end{equation}
	Hence, 
	\begin{equation}
		\begin{aligned}
			E_0(\tilde{X_n}^2) &= \ker \left(  \left( \restr{X_n}{E_0(\tilde{X_n}^2)} - \frac{\alpha_n + \alpha_n'}{2} - i \frac{\beta_n + \beta_n'}{2}\right)^2    \right) \\  
			& = V_{\frac{\alpha_n + \alpha_n'}{2} + i \frac{\beta_n + \beta_n'}{2}} \left( \restr{X_n}{E_0(\tilde{X_n}^2)}\right) \,.
		\end{aligned}
	\end{equation}
	Recall from Proposition \ref{prop:tilde_X_n^2} that $\tilde{X_n}^2$ is normal so that $\dsum_{\rho} E_\rho(\tilde{X_n}^2)$ is an orthogonal decomposition of the domain of $X_n$. For $\rho \neq 0$, the $\lambda_{+}(\rho), \lambda_{-}(\rho)$ are distinct and not equal to $\frac{\alpha_n + \alpha_n'}{2} + i \frac{\beta_n + \beta_n'}{2}$. Hence, 
	\begin{equation}
		\begin{aligned}
			E_{\lambda_{+ }(\rho)} \left( \restr{X_n}{E_\rho(\tilde{X_n}^2)}\right) & = E_{\lambda_{+ }(\rho)}(X_n) \\
			E_{\lambda_{- }(\rho)} \left( \restr{X_n}{E_\rho(\tilde{X_n}^2)}\right) & = E_{\lambda_{- }(\rho)}(X_n) \\
			V_{\frac{\alpha_n + \alpha_n'}{2} + i \frac{\beta_n + \beta_n'}{2}} \left( \restr{X_n}{E_0(\tilde{X_n}^2)}\right) &= V_{\frac{\alpha_n + \alpha_n'}{2} + i \frac{\beta_n + \beta_n'}{2}}(X_n)
		\end{aligned}
	\end{equation}
	and
	\begin{equation}
		\begin{aligned}
			V_{\lambda_+(\rho)}(X_n) & = E_{\lambda_+(\rho)}(X_n) \\
			V_{\lambda_-(\rho)}(X_n) & = E_{\lambda_-(\rho)}(X_n) \,.
		\end{aligned}
	\end{equation}
\end{proof}

Note that in the proof of Proposition \ref{prop:x_eigenspaces}, we could also start with the decomposition of the space into $V_{\rho}(X_n)$. Then, 
\begin{equation}
	\begin{aligned}
		V_{\rho}(X_n) & = \ker((X_n - \rho)^n) 
		\\ & \subset \ker\left(  \left(  \tilde{X_n}^2 -  \left(\rho - \frac{\alpha_n + \alpha_n'}{2} - i \frac{\beta_n + \beta_n'}{2} \right)^2  \right)^n    \right) 
		\\ & = \ker\left(  \tilde{X_n}^2 -  \left(\rho - \frac{\alpha_n + \alpha_n'}{2} - i \frac{\beta_n + \beta_n'}{2} \right)^2    \right) \,.
	\end{aligned}
\end{equation}
For $\rho' = (\alpha_n + \alpha_n') + i (\beta_n + \beta_n') - \rho$, $V_{\rho'}(X_n)$ is contained in the same eigenspace of $\tilde{X_n}^2$. For $\rho \neq \frac{\alpha_n + \alpha_n'}{2} + i \frac{\beta_n + \beta_n'}{2}$, we can examine the restriction of $\tilde{X_n}^2$ on $V_\rho(X_n) + V_{\rho'}(X_n)$ to see that $V_\rho(X_n) = E_\rho(X_n)$ and $V_{\rho'}(X_n) = E_{\rho'}(X_n)$.

Since $\tilde{X_n}^2$ is normal (Proposition \ref{prop:tilde_X_n^2}), then $\dsum_{\rho} E_\rho(\tilde{X_n}^2)$ is an orthogonal decomposition of the domain of $z - X_n$. Thus, in order to bound $\sigma_{\min}(z - X_n)$ from below, it suffices to bound $\sigma_{\min}\left( \restr{(z - X_n)}{E_\rho (\tilde{X_n}^2)}\right) $ from below. No meaningful bound can be given when $z$ is an eigenvalue of $X_n$, as this is exactly when $\sigma_{\min}(z - X_n) = 0$. Hence, it suffices to consider when $z$ is not an eigenvalue of $X_n$. Then, the following Proposition bounds $\sigma_{\min}\left( \restr{(z - X_n)}{E_\rho (\tilde{X_n}^2)}\right) $ from below when $z$ is not an eigenvalue of $X_n$. The key idea is to consider $z - X_n$ on invariant subspaces with small dimension.

\begin{proposition}
	\label{prop:singular_value_estimate}
	Let $z \in \C$. Consider $\rho \in \C$ that is an eigenvalue of $\tilde{X_n}^2$. For $\rho \neq 0$, let 
	\begin{equation}
		\begin{aligned}
			\lambda_{+}(\rho) & = \frac{\alpha_n + \alpha_n'}{2} + i \frac{\beta_n + \beta_n'}{2} + \sqrt{\rho} \\
			\lambda_{-}(\rho) & = \frac{\alpha_n + \alpha_n'}{2} + i \frac{\beta_n + \beta_n'}{2} - \sqrt{\rho} \,,
		\end{aligned}
	\end{equation}
	where $\sqrt{\rho}$ is a chosen square root of $\rho$. 
	
	Suppose that $z$ is not an eigenvalue of  $\restr{X_n}{E_\rho (\tilde{X_n}^2)}$. Then, the minimum singular value of $\restr{(z - X_n)}{E_\rho (\tilde{X_n}^2)}$ is bounded from below by the following:
	\begin{equation}
		\sigma_{\min} \left( \restr{(z - X_n)}{E_\rho (\tilde{X_n}^2)}  \right) \geq 
		\begin{dcases}
			\frac{ \min\left( \Abs{z - \lambda_{+}(\rho)}, \Abs{z - \lambda_{-}(\rho)}    \right)^2    }{\norm{\restr{(z - X_n)}{E_\rho (\tilde{X_n}^2)}} } & \rho \neq 0 \\
			\frac{ \Abs{z - \frac{\alpha_n + \alpha_n'}{2} - i \frac{\beta_n + \beta_n'}{2}}^2      }{ \norm{\restr{(z - X_n)}{E_0 (\tilde{X_n}^2)}}  } & \rho = 0 \,.
		\end{dcases}
	\end{equation}
\end{proposition}
\begin{proof}
	Recall that 
	\begin{equation}
		\sigma_{\min} \left( \restr{(z - X_n)}{E_\rho (\tilde{X_n}^2)}  \right) = \min_{\substack{\zeta \in E_\rho (\tilde{X_n}^2)  \\ \zeta \neq 0}} \frac{ \norm{(z - X_n)(\zeta)}  }{\norm{\zeta}} \,.
	\end{equation}
	First, consider if $\rho \neq 0$. From Proposition \ref{prop:x_eigenspaces}, $E_\rho(\tilde{X_n}^2) =  E_{\lambda_{+}(\rho)}(X_n) + E_{\lambda_{-}(\rho)}(X_n)$. Let $V = E_{\lambda_{+}(\rho)}(X_n)$, $W = E_{\lambda_{-}(\rho)}(X_n)$. Then,
	\begin{equation}
		\min_{\substack{\zeta \in E_\rho (\tilde{X_n}^2)  \\ \zeta \neq 0}} \frac{ \norm{(z - X_n)(\zeta)}  }{\norm{\zeta}} = \min_{\substack{v \in V, w \in  W \\ (v, w) \neq (0, 0)}  } \frac{ \norm{(z - X_n)(v + w)}  }{\norm{v + w}} \,.
	\end{equation}
	Fix some $v \in V, w \in  W$. Let $U = \text{span}(v, w)$. Then, $z - X_n$ fixes $U$. We proceed to show that 
	\begin{equation}
		\label{eqn:prop:singular_value_estimate:2}
		\sigma_{\min} \left(  \restr{(z - X_n )}{U} \right) \geq \frac{ \min\left( \Abs{z - \lambda_{+}(\rho)}, \Abs{z - \lambda_{-}(\rho)}    \right)^2    }{\norm{\restr{(z - X_n )}{U}}} \,.
	\end{equation}
	Note that $\norm{\restr{(z - X_n )}{U}} \neq 0$ since $z$ is not an eigenvalue of $\restr{(z - X_n )}{U}$. If either $v = 0$ or $w = 0$, then $\restr{(z - X_n )}{U} = z - \lambda_{+}(\rho)$ or $\restr{(z - X_n )}{U} = z - \lambda_{-}(\rho)$. Without loss of generality assume that $\restr{(z - X_n )}{U} = z - \lambda_{+}(\rho)$. Then, the following verifies (\ref{eqn:prop:singular_value_estimate:2}):
	\begin{equation}
		\begin{aligned}
			\sigma_{\min} \left(  \restr{(z - X_n )}{U} \right) & = \Abs{z - \lambda_{+}(\rho) } \\
			& = \frac{ \Abs{z - \lambda_{+}(\rho)}^2    }{ \norm{\restr{(z - X_n )}{U}}} \\
			& \geq \frac{ \min\left( \Abs{z - \lambda_{+}(\rho)}, \Abs{z - \lambda_{-}(\rho)}    \right)^2    }{\norm{\restr{(z - X_n )}{U}}} \,.
		\end{aligned}
	\end{equation}
	If $v \neq 0$ and $w \neq 0$ then $\dim(U) = 2$ and the eigenvalues of $\restr{(z - X_n )}{U}$ are $z - \lambda_{+}(\rho), z - \lambda_{-}(\rho)$. Since the absolute value of the product of the eigenvalues of $\restr{(z - X_n )}{U}$ is equal to the product of the singular values of $\restr{(z - X_n )}{U}$, then 
	\begin{equation}
		\begin{aligned}
			\Abs{z - \lambda_{+}(\rho)} \Abs{z - \lambda_{-}(\rho)} & = \sigma_1 \left( \restr{(z - X_n )}{U} \right)  \sigma_2 \left(  \restr{(z - X_n )}{U} \right)  \\
			& = \norm{\restr{(z - X_n )}{U}} \sigma_{\min} \left(  \restr{(z - X_n )}{U} \right) \,.
		\end{aligned}
	\end{equation}
	Hence, 
	\begin{equation}
		\begin{aligned}
			\sigma_{\min} \left(  \restr{(z - X_n )}{U} \right) & = \frac{\Abs{z - \lambda_{+}(\rho)} \Abs{z - \lambda_{-}(\rho)} }{\norm{\restr{(z - X_n )}{U}} } \\
			& \geq \frac{ \min\left( \Abs{z - \lambda_{+}(\rho)}, \Abs{z - \lambda_{-}(\rho)}    \right)^2    }{\norm{\restr{(z - X_n )}{U}}} \,.
		\end{aligned}
	\end{equation}
	Thus, in all cases of $\dim(U)$, (\ref{eqn:prop:singular_value_estimate:2}) holds. The following inequalities complete the proof in the case where $\rho \neq 0$: 
	\begin{equation}
		\begin{aligned}
			\min_{\substack{v \in V, w \in  W \\ (v, w) \neq (0, 0)}  } \frac{ \norm{(z - X_n)(v + w)}  }{\norm{v + w}} 
			& \geq \min_{\substack{v \in V, w \in  W \\ (v, w) \neq (0, 0)}  } \min_{ \substack{\zeta \in U \\ \zeta \neq 0}  }   \frac{ \norm{(z - X_n)(\zeta)}  }{\norm{\zeta}}
			\\ &= \min_{\substack{v \in V, w \in  W \\ (v, w) \neq (0, 0)}  } \sigma_{\min} \left(  \restr{(z - X_n )}{U} \right)
			\\ &\geq \min_{\substack{v \in V, w \in  W \\ (v, w) \neq (0, 0)}  } \frac{ \min\left( \Abs{z - \lambda_{+}(\rho)}, \Abs{z - \lambda_{-}(\rho)}    \right)^2    }{\norm{\restr{(z - X_n )}{U}}}
			\\ & \geq \frac{ \min\left( \Abs{z - \lambda_{+}(\rho)}, \Abs{z - \lambda_{-}(\rho)}    \right)^2    }{ \norm{\restr{(z - X_n)}{E_\rho (\tilde{X_n}^2)}}} \,.
		\end{aligned}
	\end{equation}
	Next, consider if $\rho = 0$. For $\zeta \in E_0(\tilde{X_n}^2)$, let $V = \text{span}(\zeta, \tilde{X_n} \zeta)$. Since $E_0(\tilde{X_n}^2) = \ker(\tilde{X_n}^2)$, then $\tilde{X_n}$ fixes $V$, so $z - X_n$ also fixes $V$. We proceed to show that: 
	\begin{equation}
		\label{eqn:prop:singular_value_estimate:3}
		\sigma_{\min} \left(  \restr{(z - X_n )}{V} \right) = \frac{ \Abs{z - \frac{\alpha_n + \alpha_n'}{2} - i \frac{\beta_n + \beta_n'}{2}}^2      }{\norm{\restr{(z - X_n )}{V}}  } \,.
	\end{equation}
	Note that $\norm{\restr{(z - X_n )}{V}} \neq 0$ since $z$ is not an eigenvalue of $\restr{(z - X_n )}{V}$. If $\tilde{X_n} \zeta = 0$, then $\restr{(z - X_n )}{V} = z - \frac{\alpha_n + \alpha_n'}{2} - i \frac{\beta_n + \beta_n'}{2}$. In this case, the equality clearly holds.
	
	If $\tilde{X_n} \zeta \neq 0$, then $\dim(V) = 2$. The only eigenvalue of $\restr{(z - X_n )}{V}$ is $z - \frac{\alpha_n + \alpha_n'}{2} - i \frac{\beta_n + \beta_n'}{2}$. Since the absolute value of the product of the eigenvalues of $\restr{(z - X_n )}{V}$ is equal to the product of the singular values of $\restr{(z - X_n )}{V}$, then 
	\begin{equation}
		\begin{aligned}
			\Abs{z - \frac{\alpha_n + \alpha_n'}{2} - i \frac{\beta_n + \beta_n'}{2}}^2 
			& = \sigma_1(\restr{(z - X_n )}{V}) \sigma_2(\restr{(z - X_n )}{V}) \\
			& = \norm{\restr{(z - X_n )}{V}} \sigma_{\min} \left(  \restr{(z - X_n )}{V} \right) \,.
		\end{aligned}
	\end{equation}
	Rearranging this produces (\ref{eqn:prop:singular_value_estimate:3}). Hence, in all cases of $\dim(V)$, (\ref{eqn:prop:singular_value_estimate:3}) holds. The following inequalities complete the proof in the case where $\rho = 0$: 
	\begin{equation}
		\begin{aligned}
			\min_{\substack{\zeta \in E_\rho (\tilde{X_n}^2)  \\ \zeta \neq 0}} \frac{ \norm{(z - X_n)(\zeta)}  }{\norm{\zeta}} 
			& \geq \min_{\substack{\zeta \in E_\rho (\tilde{X_n}^2)  \\ \zeta \neq 0}  } \min_{\substack{v \in V \\ v \neq 0}}   \frac{ \norm{(z - X_n)(v)}  }{\norm{v}} \\
			& =  \min_{ \substack{\zeta \in E_\rho (\tilde{X_n}^2)  \\ \zeta \neq 0}  } \sigma_{\min} (\restr{(z - X_n )}{V} ) \\
			& = \min_{ \substack{\zeta \in E_\rho (\tilde{X_n}^2)  \\ \zeta \neq 0}  } \frac{ \Abs{z - \frac{\alpha_n + \alpha_n'}{2} - i \frac{\beta_n + \beta_n'}{2}}^2      }{\norm{\restr{(z - X_n )}{V}}  } \\
			& \geq \frac{ \Abs{z - \frac{\alpha_n + \alpha_n'}{2} - i \frac{\beta_n + \beta_n'}{2}}^2      }{ \norm{\restr{(z - X_n)}{E_0 (\tilde{X_n}^2)}} } \,.
		\end{aligned}
	\end{equation}
	Both of the inequalities in the above chain are actually equalities, but we will not need this fact.
\end{proof}

We conclude by unifying the cases $\rho \neq 0$ and $\rho = 0$ in Proposition \ref{prop:singular_value_estimate} and presenting a bound on the entire domain of $z - X_n$ for all $z \in \C$: 

\begin{theorem}
	\label{thm:singular_value_estimate}
	Let $z \in \C$. The minimum singular value of $z - X_n$ satisfies the following inequality: 
	\begin{equation}
		\sigma_{\min}(z - X_n) \geq \frac{\dist(z, H_n \cap R_n)^2}{ \norm{z - X_n}  } \,.
	\end{equation}
\end{theorem}
\begin{proof}
	If $z$ is an eigenvalue of $X_n$, then the left-hand side of the inequality is $0$ and from Proposition \ref{prop:x_eigenvalue} the right-hand side of the inequality is also $0$. Thus, we may assume that $z$ is not an eigenvalue of $X_n$.
	
	Consider $\rho \in \C$ that is an eigenvalue for $\tilde{X_n}^2$. Since $z$ is not an eigenvalue for $X_n$, then $z$ is not an eigenvalue of $\restr{X_n}{E_\rho (\tilde{X_n}^2)}$. Hence, from Proposition \ref{prop:singular_value_estimate}, 
	\begin{equation}
		\sigma_{\min} \left( \restr{(z - X_n)}{E_\rho (\tilde{X_n}^2)}  \right) \geq 
		\begin{dcases}
			\frac{ \min\left( \Abs{z - \lambda_{+}(\rho)}, \Abs{z - \lambda_{-}(\rho)}    \right)^2    }{\norm{\restr{(z - X_n)}{E_\rho (\tilde{X_n}^2)}} } & \rho \neq 0 \\
			\frac{ \Abs{z - \frac{\alpha_n + \alpha_n'}{2} - i \frac{\beta_n + \beta_n'}{2}}^2      }{ \norm{\restr{(z - X_n)}{E_0 (\tilde{X_n}^2)}}  } & \rho = 0 \,.
		\end{dcases}
	\end{equation}
	Now, we proceed to make these estimates independent of $\rho$.
	
	If $\rho \neq 0$, recall from Proposition \ref{prop:x_eigenspaces} that $E_\rho(\tilde{X_n}^2) = E_{\lambda_{+}(\rho)}(X_n) + E_{\lambda_{-}(\rho)}(X_n)$. Since $\rho$ is an eigenvalue of $\tilde{X_n}^2$ then at least one of $\lambda_{+}(\rho), \lambda_{-}(\rho)$ is an eigenvalue for $\restr{X_n}{E_\rho(\tilde{X_n}^2)}$. From Proposition \ref{prop:x_eigenvalue}, any eigenvalue of $X_n$ is on $H_n \cap R_n$. Hence,
	\begin{equation}
		\begin{aligned}
			\dist(z, H_n \cap R_n) & \leq \min\left( \Abs{z - \lambda_{+}(\rho)}, \Abs{z - \lambda_{-}(\rho)} \right) \\
			& \leq \norm{\restr{(z - X_n)}{E_\rho (\tilde{X_n}^2)}} \\
			& \leq \norm{z - X_n} \,.
		\end{aligned}
	\end{equation} 
	Since $z$ is not an eigenvalue of $X_n$, then $\norm{z - X_n} \neq 0$ and thus the following inequalities hold:
	\begin{equation}
		\frac{ \min\left( \Abs{z - \lambda_{+}(\rho)}, \Abs{z - \lambda_{-}(\rho)}    \right)^2    }{\norm{\restr{(z - X_n)}{E_\rho (\tilde{X_n}^2)}} } 
		\geq \frac{ \dist(z, H_n \cap R_n) ^2    }{\norm{z - X_n }} \,.
	\end{equation}
	If $\rho = 0$, recall from Proposition \ref{prop:x_eigenspaces} that $E_0(\tilde{X_n}^2) = V_{\frac{\alpha_n + \alpha_n'}{2} + i \frac{\beta_n + \beta_n'}{2}}(X_n)$ so that $\frac{\alpha_n + \alpha_n'}{2} + i \frac{\beta_n + \beta_n'}{2}$ is an eigenvalue for $\restr{X_n}{E_0(\tilde{X_n}^2)}$. From Proposition \ref{prop:x_eigenvalue}, any eigenvalue of $X_n$ is on $H_n \cap R_n$. Hence,  
	\begin{equation}
		\begin{aligned}
			\dist(z, H_n \cap R_n) & \leq \Abs{z - \frac{\alpha_n + \alpha_n'}{2} - i \frac{\beta_n + \beta_n'}{2}} \\
			& \leq \norm{  \restr{(z - X_n)}{E_0(\tilde{X_n}^2)} } \\
			& \leq \norm{z - X_n} \,.
		\end{aligned}
	\end{equation}
	Since $z$ is not an eigenvalue of $X_n$, then $\norm{z - X_n} \neq 0$ and thus the following inequalities hold:
	\begin{equation}
		\frac{ \Abs{z - \frac{\alpha_n + \alpha_n'}{2} - i \frac{\beta_n + \beta_n'}{2}}^2      }{ \norm{\restr{(z - X_n)}{E_\rho (\tilde{X_n}^2)}}  } \geq \frac{\dist(z, H_n \cap R_n)^2}{\norm{z - X_n}} \,.
	\end{equation}
	We conclude that for any $\rho$ that is an eigenvalue for $\tilde{X_n}^2$, 
	\begin{equation}
		\sigma_{\min} \left( \restr{(z - X_n)}{E_\rho (\tilde{X_n}^2)}  \right) \geq \frac{\dist(z, H_n \cap R_n)^2}{\norm{z - X_n}} \,.
	\end{equation}
	Since $\tilde{X_n}^2$ is normal (Proposition \ref{prop:tilde_X_n^2}), then $\dsum_{\rho} E_\rho(\tilde{X_n}^2)$ is an orthogonal decomposition of the domain of $z - X_n$. Hence, 
	\begin{equation}
		\begin{aligned}
			\sigma_{\min}(z - X_n) & = \min_{\{\rho : E_\rho(\tilde{X_n}^2) \neq \{0\}\}} \sigma_{\min} \left( \restr{(z - X_n)}{E_\rho (\tilde{X_n}^2)}  \right) \\
			& \geq \min_{\{\rho : E_\rho(\tilde{X_n}^2) \neq \{0\}\}} \frac{\dist(z, H_n \cap R_n)^2}{\norm{z - X_n}} \\
			& = \frac{\dist(z, H_n \cap R_n)^2}{\norm{z - X_n}}			 \,.
		\end{aligned}
	\end{equation}
\end{proof}

\section{Proofs of Convergence and Converse}
\label{sec:convergence_converse}

In this section, we complete the proof of Theorem \ref{thm:convergence_brown_measure} and then deduce Theorem \ref{thm:converse}. 

\begin{theorem}
	\label{thm:convergence_brown_measure}
	Consider the random matrix model $X_n = P_n + i Q_n$, where $P_n, Q_n \in M_n(\C)$ are independently Haar-rotated Hermitian matrices with distributions with at most $2$ atoms. Suppose that the law of $P_n$ converges to the law of $p$ and the law of $Q_n$ converges to the law of $q$, where $p, q \in (M, \tau)$ are freely independent. Then, the empirical spectral distribution of $X_n$ converges almost surely in the vague topology to the Brown measure of $X = p + i q$.  
\end{theorem}

\begin{proof}
	Recall from the discussion after the statement of Theorem \ref{thm:convergence_brown_measure} in the Introduction that $\mu_p$ and $\mu_q$ have at most $2$ points in their supports. First, we consider the case where $p$ and $q$ have $2$ atoms, i.e. $(a, b) \in (0, 1)$ and $\alpha \neq \alpha'$, $\beta \neq \beta'$. Recall from the discussion after Definition \ref{def:p_q_spectra} that we may assume that $\alpha_n \to \alpha$, $\beta_n \to \beta$, $\alpha_n' \to \alpha'$, $\beta_n' \to \beta'$, $a_n \to a$, and $b_n \to b$
	
	Recall that we need to complete the following steps to complete the proof of Theorem \ref{thm:convergence_brown_measure} when $p$ and $q$ have $2$ atoms: 
\begin{enumerate}
	\item Show that for Lebesgue almost every $z \in \C$, the spectral measures $\nu_{n, z}$ converge to $\nu_z$ almost surely in the vague topology.
	\item Bound the minimum singular value of $z - X_n$ from below to justify the convergence
	\begin{equation}
		\frac{1}{2} \int_{0}^{\infty} \log x \, d \nu_{n, z}(x) \to \frac{1}{2} \int_{0}^{\infty} \log x \, d \nu_{z}(x) 
	\end{equation}
	for Lebesgue almost every $z \in \C$ almost surely. 
\end{enumerate}
	
	The first step was completed in Corollary \ref{cor:nu_z_convergence}.

	For the second step, it suffices to prove the convergence of the logarithmic integrals for $z \not \in H \cap R$, where $H$ and $R$ are the hyperbola and rectangle associated with $X = p + i q$. Fix a $z \not \in H \cap R$. Let $H_n$ and $R_n$ be the hyperbola and rectangle associated with $X_n$. 
	
	From the triangle inequality, for arbitrary $z \in \C$, the following inequality holds: 
	\begin{equation}
		\dist(z, H \cap R) \leq \dist(z, H_n \cap R_n) + \sup_{w \in H_n \cap R_n} \dist(w, H \cap R) \,.
	\end{equation}
	
	Hence, 
	\begin{equation}
		\dist(z, H_n \cap R_n) \geq \dist(z, H \cap R) - \sup_{w \in H_n \cap R_n} \dist(w, H \cap R) \,.
	\end{equation}
	
	As $n \to \infty$, $\sup_{w \in H_n \cap R_n} \dist(w, H \cap R) \to 0$. This can be seen by using the parameterization of $H_n \cap R_n$, $H \cap R$ in \cite[Corollary 4.5]{BrownMeasurePaper1} and comparing the points on $H_n \cap R_n$, $H \cap R$ with the same $\theta$. Since $z \not \in H \cap R$, then $\dist(z, H \cap R) > 0$. So, for $n$ sufficiently large, $\dist(z, H_n \cap R_n) > \delta$ for some $\delta > 0$. 
	
	Since $\norm{X_n} \leq \norm{P_n} + \norm{Q_n} = \max(\alpha_n, \alpha_n') + \max(\beta_n, \beta_n')$ and the sequences all converge, then the $X_n$ are uniformly bounded. Hence, the $z - X_n$ are uniformly bounded.
	
	Combining these two facts and Theorem \ref{thm:singular_value_estimate}, then for $n$ sufficiently large, $\sigma_{\min}(z - X_n) > \delta$ almost surely, for some $\delta > 0$. As $H_z(X_n)$ are uniformly bounded in $n$, then for $n$ sufficiently large, $\nu_{n, z}$ is almost surely supported on $[\delta, M]$ for some $M > 0$. 
	
	By applying the convergence $\nu_{n, z} \to \nu_z$ in the vague topology to $f \in C_c([0, \infty))$ where $f \equiv 0$ on $[\delta, M]$, we see that $\nu_z$ is also supported on $[\delta, M]$.
	
	Hence, we may choose $f \in C_c([0, \infty))$ such that $f(x) = \log(x)$ on $[\delta, M]$, so that the logarithmic integrals converge almost surely: 
	\begin{equation}
		\begin{aligned}
			& \frac{1}{2} \int_{0}^{\infty} \log x \, d \nu_{n, z}(x) = \frac{1}{2} \int_{0}^{\infty} f(x) \, d \nu_{n, z}(x) \longrightarrow \\
			& \quad \qquad  \frac{1}{2} \int_{0}^{\infty} f(x) \, d \nu_{z}(x) =\frac{1}{2} \int_{0}^{\infty} \log x \, d \nu_{z}(x) \,. 
		\end{aligned}
	\end{equation}	
	This completes the proof in the case where $p$ and $q$ both have $2$ atoms. 
	
	Next, we consider the case where both $p, q \in \R$. By applying the convergence $\mu_{P_n} \to \mu_p = \delta_p$ on $f$ that is $1$ on small a neighborhood of $p$ and supported on a slightly larger neighborhood, we see that the sum of the weights of the atoms of $P_n$ in any neighborhood $U$ converges to $1$ as $n \to \infty$. The same result applies for $Q_n$ and $q$ with neighborhood $V$. Let $n$ be sufficiently large so that the sum of the weights of the atoms of $P_n$ in $U$ is larger than $1 - \epsilon$ and the sum of the weights of the atoms of $Q_n$ in $V$ is larger than $1 - \epsilon$. If $\alpha_n, \alpha_n' \in U$ and $\beta_n, \beta_n' \in V$, then $\mu_n $ is supported on $U \times V$, so $\mu_n(U \times V) = 1$ and it is clear that $\mu_n$ tends to $\delta_{p + i q}$. In the other case, suppose that $\alpha_n \in U$ but $\alpha_n' \not \in U$. Then, $a_n = \tau(\chi_{\{ \alpha_n  \}}(P_n)) > 1 - \epsilon$. 
	
	When $(\alpha_n, \beta_n) \neq (\alpha_n', \beta_n')$, we claim the following equalities of (generalized) eigenspaces:
	\begin{equation}
		V_{\alpha_n + i \beta_n}(X_n) = E_{\alpha_n + i \beta_n}(X_n) = E_{\alpha_n}(P_n) \cap E_{\beta_n}(Q_n) \,.
	\end{equation}

	The first equality follows from Proposition \ref{prop:x_eigenspaces} and the fact that $(\alpha_n, \beta_n) \neq (\alpha_n', \beta_n')$. For the second equality, the non-trivial containment is: $E_{\alpha_n + i \beta_n}(X_n) \subset E_{\alpha_n}(P_n) \cap E_{\beta_n}(Q_n)$. For this, let $v \in E_{\alpha_n + i \beta_n}(X_n)$, $\norm{v} = 1$. Then, 
	\begin{equation}
		\alpha_n + i \beta_n = \brackets{X_n v, v} = \brackets{P_n v, v} + i \brackets{Q_n v, v} \,.
	\end{equation}
	
	As $P_n, Q_n$ are Hermitian, then $\brackets{P_n v, v}  = \alpha_n$ and $\brackets{Q_n v, v} = \beta_n$. As $P_n, Q_n$ are Hermitian and have up to $2$ eigenvalues, then $v \in E_{\alpha_n}(P_n) \cap E_{\beta_n}(Q_n)$, as desired. 
	
	Thus, from the parallelogram law,
	\begin{equation}
		\begin{aligned}
			\mu_n(\{\alpha_n + i \beta_n\}) & = \tau(\chi_{\{ \alpha_n  \}}(P_n) \wedge \chi_{ \{\beta_n\}} (Q_n))  \\
			&= \tau(\chi_{\{ \alpha_n  \}}(P_n)) + \tau(\chi_{ \{\beta_n\}} (Q_n)) - \tau(\chi_{\{ \alpha_n  \}}(P_n) \vee \chi_{ \{\beta_n\}} (Q_n)) \\
			& > (1 - \epsilon) + \tau(\chi_{ \{\beta_n\}} (Q_n)) - 1 \\
			& = \tau(\chi_{ \{\beta_n\}} (Q_n)) - \epsilon \\
			& = \mu_{Q_n}(\{\beta_n\}) - \epsilon \,.
		\end{aligned}
	\end{equation}
	Similarly, $\mu_n(\{\alpha_n + i \beta_n'\}) > \mu_{Q_n}(\{\beta_n'\}) - \epsilon$. Then, the $\mu_n$-measure of the atom(s) of $\{ \alpha_n + i \beta_n, \alpha_n + i \beta_n'  \}$ in $U \times V$ has measure greater than $1 - 3 \epsilon$, so again $\mu_n$ tends to $\delta_{p + i q}$.
	
	Finally, the case where exactly one of $p$ and $q$ is constant follows in a similar manner. Suppose that $q$ has $2$ atoms and $p \in \R$. Then $b \in (0, 1)$, $\beta \neq \beta'$, $b_n \to b$, $\beta_n \to \beta$, and $\beta_n' \to \beta'$. For the case that both atoms of $P_n$ are in a small neighborhood of $p$, the branches of $H_n \cap R_n$ are in small neighborhoods of $\{p + i \beta,  p + i \beta'\}$, and from \cite[Proposition 4.7]{BrownMeasurePaper1} these branches have the appropriate measures $\mu_{Q_n}(\{\beta_n\}) = b_n$ and $\mu_{Q_n}( \{\beta_n'\} ) = 1 - b_n$ so that $\mu_n$ tends towards $\mu = b \delta_{ \beta } + (1 - b) \delta_{\beta'}$. For the case that $\alpha_n$ is in a small neighborhood of $p$ and $\mu_{P_n}( \{\alpha_n\}) \approx 1$, then we use the previous argument to show that the $\mu_n (\{\alpha_n + i \beta_n\} ) \approx \mu_{Q_n}( \{\beta_n\} ) $ and $\mu_n( \{\alpha_n + i \beta_n'\} ) \approx \mu_{Q_n}(\{\beta_n'\})$.
\end{proof}

\begin{theorem}
	\label{thm:converse}
	Consider the random matrix model $X_n = P_n + i Q_n$, where $P_n, Q_n \in M_n(\C)$ are independently Haar-rotated Hermitian matrices with distributions with at most $2$ atoms. Suppose that the empirical spectral distribution of $X_n$ converges in probability in the vague topology to some deterministic probability measure $\mu$. Then, the law of $P_n$ converges to the law of $p$ and the law of $Q_n$ converges to the law of $q$ for some $p, q \in (M, \tau)$ where $p$ and $q$ are freely independent. Hence, $\mu$ is the Brown measure of $X = p + i q$. 
\end{theorem}
\begin{proof}
	The majority of the proof is spent showing that if $\mu_n$ converges vaguely in probability to a deterministic probability measure $\mu$, then $\mu_{P_n}$ and $\mu_{Q_n}$ are tight, i.e. for any subsequence of these sequences, there exists a further subsequence that converges in the vague topology. 
	
	Having shown this, the conclusion follows by passing to subsequences, using Theorem \ref{thm:convergence_brown_measure}, and using that the Brown measure of $X = p + i q$ determines $\mu_p$ and $\mu_q$ (Corollary \ref{cor:brown_measure_injective}). 
	
	Now, it suffices to show a subsequence of $\mu_n$ has $\mu_{P_{n_k}}$ and $\mu_{Q_{n_k}}$ converge.
	
	First, if the atoms of $P_n$ and $Q_n$, $\alpha_n, \alpha_n', \beta_n, \beta_n'$ are uniformly bounded then we may extract a convergent subsequence where all atoms and weights $a_n, b_n \in [0, 1]$ converge. Then, it is clear that $P_{n_k}$ and $Q_{n_k}$ converge in the vague topology.
	
	Thus, we may assume that $(\alpha_n, \alpha_n', \beta_n, \beta_n') \in \R^4$ is not uniformly bounded. By passing to a subsequence, we may assume that $\norm{(\alpha_n, \alpha_n', \beta_n, \beta_n')} \to \infty$. Also assume without loss of generality that $\Abs{\alpha_n} = \max( \Abs{\alpha_n}, \Abs{\alpha_n'}, \Abs{\beta_n}, \Abs{\beta_n'})$. We proceed to show that $a_n$, the weight of $\alpha_n$ in $\mu_{P_n}$, converges to $0$. 
	
	Let $X_n' = X_n / \alpha_n$, so that $X_n' = P_n' + i Q_n'$, where $P_n' = P_n / \alpha_n$, $Q_n' = Q_n / \alpha_n$. Then, the atoms of $P_n'$ are at $1, \alpha_n' / \alpha_n$ with respective weights $a_n, 1 - a_n$ and the atoms of $Q_n'$ are at $\beta_n / \alpha_n, \beta_n' / \alpha_n$ with respective weights $b_n, 1 - b_n$. By construction, the atoms of $P_n', Q_n'$ are in $[-1, 1]$. Let $\mu_n'$ be the empirical spectral distribution of $X_n'$.
	
	Since $\mu$ is a probability measure and $\mu_n \to \mu$ in probability in the vague topology, then we deduce that for every $\epsilon > 0$ and $\delta > 0$ there exists a compact set $K_\epsilon \subset \C$ such that $\mu(K_\epsilon) > 1 - \epsilon$ and for $n$ sufficiently large, $\mathbb{P}( \mu_n(K_\epsilon) > 1 - \epsilon ) > 1 - \delta $.
	
	Now, note that $\mu_{n}'( \frac{1}{\alpha_n} K_\epsilon  ) = \mu_n(K_\epsilon) $. For any neighborhood $U$ of $0$ in $\C$, we may choose $n$ sufficiently large to that $\frac{1}{\alpha_n} K_\epsilon \subset U$. Hence, for $n$ sufficiently large, $\mathbb{P}( \mu_n'(U) > 1 - \epsilon ) > 1 - \delta$. This implies that $\mu_n'$ converges to $\delta_0$ in probability. 
	
	From the boundedness of the atoms if $P_n', Q_n'$, any subsequence of $X_n'$ has a convergent subsequence where the atoms and weights of $P_{n_k}', Q_{n_k}'$ converge. From Theorem \ref{thm:convergence_brown_measure}, this implies that $\mu_{n_k}'$ converges almost surely to $\mu'$, the Brown measure of $X' = p' + i q'$, where $\mu_{P_{n_k}'}$ converges to $\mu_{p'}$ and $\mu_{Q_{n_k}'} $ converges to $\mu_{q'}$. Hence, $\mu' = \delta_0$. Since the assignment $(\mu_{p'}, \mu_{q'}) \mapsto \mu'$ is 1-1, then $p = q = 0$. Hence, $\mu_{P_{n_k}} \to \delta_0$. This implies that $\mu_{P_{n}} \to \delta_0$. For this to happen, the weight of the atom at $1$, $a_n$, has to tend to $0$. Thus we conclude that if $(\alpha_n, \alpha_n', \beta_n, \beta_n') \in \R^4$ is not uniformly bounded then the weight of one of these atoms goes to $0$.
	
	Without loss of generality, let us keep assuming that $a_n \to 0$. If $\alpha_n'$ is not bounded, then by passing a subsequence where $\Abs{\alpha_n'} \to \infty$, then the atoms $\alpha_n' + i \beta_n$ and $\alpha_n' + i \beta_n'$ go to $\infty$. As the weight of $\alpha_n$ in $P_n$ tends to $1$, then from the previous argument, the measure of the atoms $\alpha_n' + i \beta_n$ and $\alpha_n' + i \beta_n'$ tends to $1$. Then the limit measure has to be $\mu = 0$, a contradiction. Hence, $\alpha_n'$ is bounded, and by passing to a subsequence we may assume that $\alpha_n' \to \alpha'$. Thus, $\mu_{P_n} \to \delta_{\alpha'}$ for some $\alpha' \in \R$. 
	
	Now, we proceed to show that $\mu_{Q_n}$ is a tight sequence of measures. Since $\mu_n \to \mu$ in probability and $\mu$ is a probability measure, then for any $\epsilon > 0$, we may choose a compact set $K_\epsilon \subset \C$ so that for all $n$ sufficiently large, $\mathbb{P}( \mu_n(K_\epsilon) > 1 - \epsilon )  > 0$. Also assume $n$ is sufficiently large so that $a_n < \epsilon$. Hence,
	\begin{equation}
		\begin{aligned}
			\mu_n(\{\alpha_n' + i \beta_n\}) & = \tau(\chi_{\{ \alpha_n'  \}}(P_n) \wedge \chi_{ \{\beta_n\}} (Q_n))  \\
			&= \tau(\chi_{\{ \alpha_n'  \}}(P_n)) + \tau(\chi_{ \{\beta_n\}} (Q_n)) - \tau(\chi_{\{ \alpha_n'  \}}(P_n) \vee \chi_{ \{\beta_n\}} (Q_n)) \\
			& = (1 - a_n) + \tau(\chi_{ \{\beta_n\}}(Q_n)) - \tau(\chi_{\{ \alpha_n'  \}}(P_n) \vee \chi_{ \{\beta_n\}} (Q_n)) \\
			& > (1 - \epsilon) + \tau(\chi_{ \{\beta_n\}} (Q_n)) - 1 \\
			& = \tau(\chi_{ \{\beta_n\}} (Q_n)) - \epsilon \\
			& = \mu_{Q_n}(\{\beta_n\}) - \epsilon \\
			&= b_n - \epsilon \,.
		\end{aligned}
	\end{equation}
	Similarly, $\mu_n(\{\alpha_n' + i \beta_n'\}) > (1 - b_n) - \epsilon$. Thus, for $\epsilon < 1/3$, at least one of $\alpha_n' + i \beta_n, \alpha_n' + i \beta_n'$ is in $K_\epsilon$. Let $C_\epsilon$ be the projection of $K_\epsilon$ onto the $y$ coordinate. If both $\alpha_n' + i \beta_n, \alpha_n' + i \beta_n' \in K_\epsilon$, then $\mu_{Q_n}(C_\epsilon) = 1$. Assume that $\alpha_n' + i \beta_n \not \in K_\epsilon$. Then, with positive probability, $\mu_n(K_\epsilon) > 1 - \epsilon$, so $\mu_n(\{\alpha_n' + i \beta_n\}) < \epsilon$. From the inequalities above, this implies that $b_n < 2 \epsilon$, so that $\mu_{Q_n}(C_\epsilon) > 1 - 2 \epsilon$. Hence, $Q_n$ is tight.
	
	Then, by passing to subsequences, using Theorem \ref{thm:convergence_brown_measure}, and using that the Brown measure of $X = p + i q$ determines $\mu_p$ and $\mu_q$ (Corollary \ref{cor:brown_measure_injective}), then we conclude that the law of $Q_n$ converges to the law of $q$. Since the law of $P_n$ converged to $\delta_{\alpha'}$, then this completes the proof.
\end{proof}

\newpage

\printbibliography

\end{document}